\documentclass[reqno]{amsart}

\usepackage[english]{babel}
\usepackage{amsmath,amssymb,amsthm}
\usepackage[hidelinks]{hyperref}
\usepackage[centering]{geometry}

\theoremstyle{plain}
\newtheorem{theorem}{Theorem}[section]
\newtheorem{lemma}[theorem]{Lemma}
\newtheorem{proposition}[theorem]{Proposition}
\newtheorem{corollary}[theorem]{Corollary}

\theoremstyle{definition}

\theoremstyle{remark}
\newtheorem{remark}[theorem]{Remark}

\numberwithin{equation}{section}

\title[Relative Liouville Rigidity]
{Relative Liouville Rigidity for Lagrangian
Self-Similar Submanifolds with Legendrian Boundary}
\thanks{}

\author{Dong Gao\and Hui Ma\and Zeke Yao}

\address{D.~Gao, School of Science, Beijing
	University of Civil Engineering and Architecture, Beijing 102616, P.R. China}
\email{gaodong@bucea.edu.cn}

\address{H.~Ma, Department of Mathematical Sciences, Tsinghua
	University, Beijing, 100084, P.R. China}
\email{ma-h@tsinghua.edu.cn}

\address{Z.~Yao, School of Mathematical Sciences, South China Normal University, Guangzhou 510631, P.R. China}
\email{yaozkleon@163.com}


\date{}

\begin{document}

	\begin{abstract}
We study compact Lagrangian self-similar immersions in the unit ball with
Legendrian boundary. The
Liouville form naturally defines a relative de
Rham class, and we prove that the vanishing of this relative Liouville
class forces the immersion to be a diffeomorphism onto an equatorial
Lagrangian disk. In particular,
this gives rigidity for exact immersions with connected boundary. We also
establish a boundary flux identity and a localized boundary unique
continuation theorem, yielding rigidity when a nonempty relatively open
boundary portion satisfies the free-boundary condition, when the cosine of the contact angle has a
fixed sign, or when the Legendrian capillary boundary is connected. In the
two-dimensional minimal case, these results recover the smooth 
disk theorem of Li--Wang--Weng and the free-boundary theorem of Luo--Sun, 
while extending the underlying rigidity mechanisms to higher
dimensions and the self-similar setting. Finally, in every dimension \(n\geq 2\), we
construct compact embedded exact Lagrangian examples in the minimal, self-shrinking, and self-expanding
cases with two Legendrian capillary boundary components, nonzero relative Liouville class, and
supplementary contact angles.
\end{abstract}

	\maketitle
	
	\section{Introduction}

Let \(\Omega\subset\mathbb R^{n+1}\) be a bounded domain with smooth
boundary. A capillary hypersurface in \(\Omega\) is a compact
constant-mean-curvature hypersurface whose boundary lies on
\(\partial\Omega\) and meets it at a constant contact angle. When the
contact angle is \(\pi/2\) and the mean curvature vanishes, it is called a
free boundary minimal hypersurface.

A classical theorem of Nitsche~\cite{Nitsche} asserts that the only
immersed free boundary minimal disk in \(\mathbb B^3\) is the equatorial
disk. Fraser and Schoen~\cite{FraserSchoen} extended this result to higher
codimensions, proving that every proper branched immersed free boundary
minimal disk in a Euclidean ball is an equatorial plane disk. We refer to
\cite{LiMM} for a survey of free boundary minimal submanifolds.

Motivated by capillary geometry, Li, Wang, and Weng~\cite{LWW} introduced
the Legendrian capillary boundary condition for Lagrangian submanifolds in
K\"ahler manifolds. They proved that a branched minimal Lagrangian disk in
\(\mathbb B^4\subset\mathbb C^2\) with Legendrian capillary boundary is an
equatorial plane disk. Luo and Sun~\cite{LS} subsequently proved that every
minimal Lagrangian surface in \(\mathbb B^4\) with Legendrian free boundary
is an equatorial plane disk, and also classified the annulus-type minimal
case with Legendrian capillary boundary. 
Relatedly, Gaia~\cite{Gaia} proved equatorial rigidity for weakly
conformal branched free-boundary Hamiltonian stationary Lagrangian disks
in \(\mathbb B^4\) with Legendrian boundary, assuming continuity of the
Lagrangian angle and a localization property.

Throughout the paper, all manifolds and submanifolds under consideration
are assumed to be connected unless otherwise stated. We write
\[
	\mathbb B^{2n}:=\{z\in\mathbb C^n:|z|<1\},
	\qquad
	\overline{\mathbb B}^{2n}:=\{z\in\mathbb C^n:|z|\leq1\}.
\]
By an \emph{equatorial Lagrangian \(n\)-disk} we mean
\[
	\Pi\cap\overline{\mathbb B}^{2n},
\]
where \(\Pi\subset\mathbb C^n\) is a Lagrangian linear \(n\)-plane.

Let
\(
	x:M^n\rightarrow\overline{\mathbb B}^{2n}
\)
be a smooth immersion of a compact manifold with nonempty boundary. We
always assume that
\[
	x(M^\circ)\subset\mathbb B^{2n},
	\qquad
	x(\partial M)\subset\mathbb S^{2n-1}.
\]
Denote by \(\nu\) the outward unit conormal of \(\partial M\) in \(M\).
The boundary is Legendrian if and only if
\[
	\nu=\sin\theta\,x+\cos\theta\,Jx
	\quad\text{on }\partial M
\]
for a function \(\theta\in[0,\pi)\), called the contact angle
\cite{LWW}. The boundary is Legendrian capillary if \(\theta\) is constant
on each connected component, and it is Legendrian free if
\(\theta=\pi/2\).

A Lagrangian immersion is called \emph{self-similar} if
\[
	H+\kappa x^\perp=0
\]
for some constant \(\kappa\in\mathbb R\). The cases
\(\kappa=0,1,-1\) correspond respectively to minimal submanifolds,
self-shrinkers, and self-expanders. Such submanifolds arise naturally in
Lagrangian mean curvature flow. Their construction was initiated by
Anciaux~\cite{Anciaux} and further developed in, among others,
\cite{JoyceLeeTsui,LeeWang1,LeeWang2}. Rigidity results for complete
Lagrangian self-similar solutions have been obtained under graphical,
curvature, and Maslov-type assumptions; see
\cite{ChauChenYuan,DingXin,ChengHoriWei,LiWangWei,LotayNeves,Neves}.
Compact rigidity questions are also closely related to characterizations
of the Clifford torus~\cite{CastroLerma2,LiWang}.

Let
\[
	\lambda_{\mathbb C^n}
	=
	\sum_{j=1}^n(x_j\,dy_j-y_j\,dx_j)
\]
be the standard \emph{Liouville form}, and set
\(\lambda=x^*\lambda_{\mathbb C^n}\). Since \(x\) is Lagrangian,
\(d\lambda=0\). If \(\iota:\partial M\hookrightarrow M\) denotes the
inclusion, the Legendrian boundary condition gives
\(\iota^*\lambda=0\). Thus \(\lambda\) determines the relative Liouville
class
\[
	\mathcal L_{\mathrm{rel}}(x)
	:=
	[\lambda]_{\mathrm{rel}}
	\in H^1_{\mathrm{dR}}(M,\partial M;\mathbb R).
\]
Our first main result identifies the vanishing of this class as the natural
rigidity condition.

\begin{theorem}
	\label{thm:relative-liouville-rigidity}
	Let
	\(
		x:M^n\rightarrow
		\overline{\mathbb B}^{2n}\subset\mathbb C^n
	\)
	be a smooth Lagrangian self-similar immersion of a compact manifold
	with Legendrian boundary on \(\mathbb S^{2n-1}\). If
	\[
		\mathcal L_{\mathrm{rel}}(x)=0
		\quad\text{in }
		H^1_{\mathrm{dR}}(M,\partial M;\mathbb R),
	\]
	then \(x\) is a diffeomorphism onto an equatorial Lagrangian
	\(n\)-disk.
\end{theorem}

If \(x\) is exact, say \(\lambda=du\), then
\(\iota^*\lambda=0\) implies that \(u\) is constant on each boundary
component. Hence connectedness of \(\partial M\) makes the relative
Liouville class vanish and gives the following direct consequence.

\begin{corollary}
	\label{thm:exact-self-similar-legendrian}
	Let
	\(
		x:M^n\to
		\overline{\mathbb B}^{2n}\subset\mathbb C^n
	\)
	be a compact exact Lagrangian self-similar immersion with connected
	Legendrian boundary on \(\mathbb S^{2n-1}\). Then \(x\) is a
	diffeomorphism onto an equatorial Lagrangian \(n\)-disk.
\end{corollary}

When \(n=2\), \(\kappa=0\), and \(M\) is a disk, its relative first
cohomology vanishes. Thus, in the smooth setting,
Theorem~\ref{thm:relative-liouville-rigidity} extends the corresponding
rigidity theorem of Li, Wang, and Weng~\cite{LWW}. It also shows that their
Legendrian capillary assumption may be replaced by the weaker Legendrian
boundary condition. More generally,
Theorem~\ref{thm:relative-liouville-rigidity} yields topological and
relative Maslov-class rigidity criteria.

The weighted coclosedness of the Liouville form leads to two further
results. The first is a local-to-global rigidity theorem that does not
require exactness or connectedness of the boundary.

\begin{theorem}
	\label{thm:local-free-boundary-rigidity}
	Let
	\(
		x:M^n\to
		\overline{\mathbb B}^{2n}\subset\mathbb C^n
	\)
	be a compact Lagrangian self-similar immersion with Legendrian boundary.
	Suppose that there exists a nonempty relatively open subset
	\(\Gamma\subset\partial M\) on which
	\(\theta=\pi/2\). Then \(x\) is a diffeomorphism onto an
	equatorial Lagrangian \(n\)-disk.
\end{theorem}

The second is a balance law for the normal component of the Liouville
form along the boundary.

\begin{theorem}
	\label{thm:boundary-liouville-flux}
	Let
	\(
		x:M^n\to
		\overline{\mathbb B}^{2n}\subset\mathbb C^n
	\)
	be a compact Lagrangian self-similar immersion with Legendrian boundary.
	Then
	\[
		\int_{\partial M}\cos\theta\,d\sigma=0.
	\]
    In particular, if 
    \(\partial M=\Gamma_1\sqcup\cdots\sqcup\Gamma_N\) and  
    \(x\) has Legendrian capillary boundary with 
      contact angle  \(\theta_j\) on each boundary component
	\(\Gamma_j\), then
	\[
	\sum_{j=1}^N|\Gamma_j|\cos\theta_j=0,
	\]
		where \(|\Gamma_j|\) denotes the induced \((n-1)\)-dimensional volume of
		\(\Gamma_j\).
\end{theorem}

Notice that the integral identity requires only the Legendrian boundary
condition. The capillary assumption enters only in the last assertion,
where \(\cos\theta\) is constant on each connected boundary component.

As immediate consequences, the same rigidity conclusion holds whenever
\(\cos\theta\) has a fixed sign on \(\partial M\); see
Corollary~\ref{cor:fixed-sign-rigidity}. In particular, if the
Legendrian capillary boundary is connected, then its contact angle must be
\(\pi/2\), and the immersion is a diffeomorphism onto an equatorial
disk. The free-boundary
case extends the theorem of Luo and Sun~\cite{LS} from minimal Lagrangian surfaces
in \(\mathbb B^4\) to arbitrary dimensions and to the self-similar
setting.

The proofs are based on the observation that the Liouville form is
naturally a weighted relative harmonic one-form:
\[
	d\lambda=0,
	\qquad
	\delta_f\lambda=0,
	\qquad
	\iota^*\lambda=0,
	\qquad
	f=\frac{\kappa}{2}|x|^2.
\]
If the relative Liouville class vanishes, then \(\lambda\) has a global
primitive with zero boundary value, and the maximum principle forces
\(\lambda\equiv0\). The boundary Liouville flux identity follows instead
from the weighted coclosed equation and the divergence theorem.

The proof of Theorem~\ref{thm:local-free-boundary-rigidity} uses a recent
theorem of Gerner~\cite{Gerner}, who established boundary unique
continuation for exterior differential forms satisfying an elliptic
differential inequality and vanishing to infinite order at a boundary
point. We verify the required hypotheses for the weighted
closed--coclosed Liouville form and obtain a localized version adapted to
a free boundary patch. Once \(\lambda\) vanishes, the position vector is
everywhere tangent to the immersion, and a cone-rigidity argument shows
that the immersion is a diffeomorphism onto an equatorial disk.

Finally, for every \(n\geq2\), \(\kappa\in\{-1,0,1\}\), and
\(a\in(0,1)\), we construct a compact embedded exact Lagrangian
self-similar submanifold with two totally geodesic Legendrian capillary
boundary components. Their contact angles are supplementary and neither
is free, while the relative Liouville class is nonzero. These examples
show that connectedness of the boundary in the exact corollary and the
fixed-sign hypothesis in the flux rigidity result cannot in general be
omitted. In the minimal case, the construction recovers compact portions
of the Lagrangian catenoids.

The two-boundary examples led us to pose a classification problem in
earlier arXiv versions of this paper, first for Lagrangian surfaces and
then in all dimensions \(n\geq2\). The higher-dimensional question, and
hence also its two-dimensional specialization, is answered in the
subsequent work of Gao--Luo--Ma--Yin~\cite{GaoLuoMaYin}, under the
stronger Legendrian capillary boundary condition. The present paper
develops the relative Liouville, boundary unique-continuation, and flux
mechanisms valid under the weaker Legendrian boundary condition, whereas
\cite{GaoLuoMaYin} establishes the corresponding global classification
within the capillary class.

The paper is organized as follows. In
Section~\ref{sec:preliminaries}, we recall the Legendrian boundary
geometry, establish the weighted Liouville identities, and prove the
Liouville cone-rigidity lemma. In Section~\ref{sec:rigidity}, we prove the
relative rigidity theorem and its exact, topological, and Maslov
consequences, establish localized boundary unique continuation, and
derive the boundary Liouville flux identity. In
Section~\ref{sec:examples}, we review the unified Anciaux construction
and construct the embedded two-boundary examples.

\section{Preliminaries}\label{sec:preliminaries}
	Let
	\(
	\mathbb C^n=\{(z_1,\ldots,z_n):z_j\in\mathbb C,\ 1\leq j\leq n\}
	\)
	denote the complex Euclidean space endowed with the Hermitian inner product
	\[
	(z,w)=\sum_{j=1}^n z_j\overline{w}_j.
	\]
	The associated Euclidean metric is given by
	\(\langle\cdot,\cdot\rangle=\Re(\cdot,\cdot)\), and the standard
	Kähler form is obtained by
	\(\omega(\cdot,\cdot)=\langle J\cdot,\cdot\rangle\), where \(J\)
	is the standard complex structure on \(\mathbb C^n\). In the standard
	coordinates \(z_j=x_j+iy_j\), one has
	\[
	\omega=\sum_{j=1}^n dx_j\wedge dy_j.
	\]
	
	An immersed submanifold \(x:M\to\mathbb C^n\) is called \emph{Lagrangian} if
	\(x^*\omega=0\), which is equivalent to the condition \(J(TM)=T^\perp M\).
    
	We shall also make use of the Liouville form on \(\mathbb C^n\) given by
	\[
	\lambda_{\mathbb C^n}
	=
	\sum_{j=1}^n (x_j\,dy_j-y_j\,dx_j).
	\]
	It satisfies \(d\lambda_{\mathbb C^n}=2\omega\). For a Lagrangian immersion
	\(x:M^n\to\mathbb C^n\), we write
	\[
	\lambda:=x^*\lambda_{\mathbb C^n}
	\]
	for the pullback of the Liouville form to \(M\). For simplicity, we shall
	also refer to \(\lambda\) as the \emph{Liouville form of the Lagrangian immersion} \(x\). Equivalently, identifying \(x\) with the position vector field along the immersion, we have
	\[
	\lambda(Y)=\omega(x,Y)=\langle Jx,Y\rangle,
	\qquad Y\in TM.
	\]
	Since \(M\) is Lagrangian, it follows that
	\[
	d\lambda=x^*(d\lambda_{\mathbb C^n})=2x^*\omega=0.
	\]
	If there exists a smooth function \(\rho\) on \(M\) such that
	\(\lambda=d\rho\), then the Lagrangian immersion \(x\) is called \emph{exact}.

	If \(x(\partial M)\subset\mathbb S^{2n-1}\) is Legendrian and
	\(\iota:\partial M\hookrightarrow M\) is the inclusion, then
	\(\iota^*\lambda=0\). Thus \(\lambda\) is a relative closed one-form
	and defines
	\[
	\mathcal L_{\mathrm{rel}}(x)
	:=[\lambda]_{\mathrm{rel}}
	\in H^1_{\mathrm{dR}}(M,\partial M;\mathbb R).
	\]
	The condition \(\mathcal L_{\mathrm{rel}}(x)=0\) is equivalent to the
	existence of \(u\in C^\infty(M)\) such that \(\lambda=du\) and
	\(u|_{\partial M}=0\). Equivalently, a primitive of \(\lambda\) takes
	one and the same constant value on all boundary components.
	
	An immersion
	\[
	\psi:\Sigma^{n-1}\to\mathbb S^{2n-1}
	\]
	is called \emph{Legendrian} if
	\[
	\langle J\psi,\psi_*Y\rangle
	=\langle J\psi_*Y,\psi_*Z\rangle=
	0,
	\qquad
	\text{for any }Y, Z\in T\Sigma.
	\]
Equivalently,
\[
J(T\Sigma)\subset T^\perp_{\mathbb S}\Sigma,
\qquad
J\psi\in\Gamma(T^\perp_{\mathbb S}\Sigma),
\]
where \(T^\perp_{\mathbb S}\Sigma\) denotes the normal bundle of
\(\psi(\Sigma)\) in \(\mathbb S^{2n-1}\). Thus, for a Legendrian
submanifold, the normal bundle of $\Sigma$ in ${\mathbb S}^{2n-1}$ decomposes as
\[
T^\perp_{\mathbb S}\Sigma
=
J(T\Sigma)\oplus \operatorname{span}\{J\psi\}.
\]
		\begin{proposition}[\cite{LWW}]\label{legendrianbdy}
		Let \(x:M\to\mathbb C^n\) be a Lagrangian submanifold with boundary
		\(x(\partial M)\subset\mathbb S^{2n-1}\). Let \(\nu\) be the outward unit
		conormal of \(\partial M\) in \(M\). Then \(x(\partial M)\) is Legendrian
		in \(\mathbb S^{2n-1}\) if and only if there exists a function
		\(\theta=\theta(x)\in[0,\pi)\) such that
		\[
		\nu=\sin\theta\,x+\cos\theta\,Jx
		\qquad
		\text{on }\partial M.
		\]
	\end{proposition}
	
	The angle \(\theta\) is called the \emph{contact angle}. If \(\theta\) is constant on each connected component of \(\partial M\), then
	\(M\) is called a \emph{Lagrangian submanifold with Legendrian capillary boundary},
	or simply \emph{capillary Lagrangian submanifold}. In the special case where
	\(\theta=\pi/2\), \(M\) is said to have \emph{Legendrian free boundary}.
	
	For a Lagrangian submanifold \(M\), the induced volume form \(d\mu_M\) and
	the complex-valued \(n\)-form
	\[
	\Omega=dz_1\wedge dz_2\wedge\cdots\wedge dz_n
	\]
	are related by
	\[
	\left.\Omega\right|_M=e^{i\beta}d\mu_M
	=
	\cos\beta\,d\mu_M+i\sin\beta\,d\mu_M,
	\]
	where 
	$
	\beta:M\to \mathbb R/2\pi\mathbb Z
	$
	is the Lagrangian angle, multi-valued in general. Locally it satisfies (see for example
	\cite{RJ})
	\[
	H=J\nabla\beta,
	\]
	where \(H\) is the mean curvature vector of \(M\). The one-form
	\(
	\alpha_H:=\langle JH,\cdot\rangle
	\)
	is well-defined and closed on \(M\), and its cohomology class
	\[
	[\alpha_H]\in H^1_{\mathrm{dR}}(M;\mathbb R)
	\]
	is called the \emph{Maslov class} of \(M\).

	We begin with the following basic identity for the Liouville form of a
	Lagrangian self-similar immersion. 
   
	\begin{lemma}
\label{lem:liouville-self-similar}
Let \(x:M^n\to\mathbb C^n\) be a Lagrangian immersion satisfying
\(H+\kappa x^\perp=0\) for some constant \(\kappa\in\mathbb R\). Then
\[
d\lambda=0,
\qquad
\alpha_H=-\kappa\lambda,
\qquad
\delta_f\lambda=0,
\]
where \(f:=\frac{\kappa}{2}|x|^2\) and
\(\delta_f\eta:=\delta\eta+\eta(\nabla f)\) for any one-form \(\eta\).
\end{lemma}
	
	\begin{proof}
		Since \(M\) is Lagrangian and \(d\lambda_{\mathbb C^n}=2\omega\),
		we have \(d\lambda=0\). If \(\kappa=0\), then \(H=0\), and the identity
\(\alpha_H=-\kappa\lambda\) is trivial. If \(\kappa\neq0\),
from the self-similar equation \(H+\kappa x^\perp=0\), we obtain for every \(Y\in TM\),
        \[
        \alpha_H(Y)
        =\langle JH,Y\rangle
        =-\kappa\langle Jx^\perp,Y\rangle
        =-\kappa\langle Jx,Y\rangle
        =-\kappa\lambda(Y).
        \]
		Thus \(\alpha_H=-\kappa\lambda\).
		
		Let \(\{e_1,\ldots,e_n\}\) be a local orthonormal frame on \(M\). Using
		the Euclidean connection \(\overline\nabla\), the Levi-Civita connection
		\(\nabla\) of \(M\), and the second fundamental form \(B\), we compute
		\[
		\begin{aligned}
			(\nabla_{e_i}\lambda)(e_i)
			&=
			e_i\langle Jx,e_i\rangle
			-
			\langle Jx,\nabla_{e_i}e_i\rangle  \\
			&=
			\langle Je_i,e_i\rangle
			+
			\langle Jx,\overline\nabla_{e_i}e_i-\nabla_{e_i}e_i\rangle  \\
			&=
			\langle Jx,B(e_i,e_i)\rangle, 
		\end{aligned}
		\]
		where we have used \(\langle Je_i,e_i\rangle=0\). Therefore
		\[
		\delta\lambda
		=
		-\sum_i(\nabla_{e_i}\lambda)(e_i)
		=
		-\langle Jx,H\rangle.
		\]
		Since \(JH\) is tangent to \(M\), we get
		\[
		\delta\lambda
		=
		\langle x,JH\rangle
		=
		\langle x^\top,JH\rangle
		=
		\alpha_H(x^\top)
		=
		-\kappa\lambda(x^\top).
		\]
		Finally, \(\nabla f=\kappa x^\top\), and hence
		\[
		\delta_f\lambda
		=
		\delta\lambda+\lambda(\nabla f)
		=
		-\kappa\lambda(x^\top)+\kappa\lambda(x^\top)
		=
		0.
		\]
	\end{proof}
	The following lemma does not use the self-similar equation.
    It shows that the vanishing of the Liouville form
	forces the immersion to be a diffeomorphism onto an equatorial Lagrangian disk.
\begin{lemma}
\label{lem:liouville-rigidity}
Let \(x:M^n\to\overline{\mathbb B}^{2n}\subset\mathbb C^n\) be a compact
smooth Lagrangian immersion satisfying
\(x(M^\circ)\subset\mathbb B^{2n},
	x(\partial M)\subset\mathbb S^{2n-1}\).
If the Liouville form \(\lambda\equiv0\), then
there exists a Lagrangian linear \(n\)-plane \(\Pi\subset\mathbb C^n\) 
such that
\[
x:M\longrightarrow\Pi\cap\overline{\mathbb B}^{2n}
\]
is a diffeomorphism.
\end{lemma}
	
	\begin{proof}
		We first show that \(x^\perp\equiv0\). Decompose $x=x^\top+x^\perp$ along \(M\). 
        Since \(\lambda\equiv0\), 
		\[
		0=\lambda(Y)=\langle Jx, Y \rangle= \langle Jx^\top,Y \rangle +\langle Jx^\perp, Y\rangle
		\]
        for every \(Y\in TM\).
		As \(x\) is Lagrangian, \( \langle Jx^\top,Y\rangle =0\), and therefore
        \(
        0=\langle Jx^\perp,Y\rangle.
        \)	
		Since  \(J\) maps $T^\perp M$ isomorphically
		onto $TM$, we have \(Jx^\perp\in TM\). Taking
		\(Y=Jx^\perp\) yields  $x^\perp\equiv0$. 
        
       Hence the position vector field is tangent to \(M\). Let
\(X\in\Gamma(TM)\) be the tangent vector field determined by
\(dx(X)=x\), and let \(\Phi_t\) denote its flow. Then
\[
\frac{d}{dt}x(\Phi_t(p))=x(\Phi_t(p)),
\]
and hence
\[
x(\Phi_t(p))=e^t x(p).
\]
Since \(|x|=1\) on \(\partial M\) and
\[
X(|x|^2)=2|x|^2>0
\]
there, \(X\) points outward along \(\partial M\). Hence the backward
flow \(\Phi_t(p)\) exists for all \(t\leq0\). Moreover,
\[
|\dot{\Phi}_t(p)|
 =|X(\Phi_t(p))|
 =|x(\Phi_t(p))|
 =e^t|x(p)|.
\]
It follows that \(\Phi_t(p)\) converges as \(t\to-\infty\) to some
\(q(p)\in M\), and
\[
x(q(p))=\lim_{t\to-\infty}e^t x(p)=0.
\]

Let \(Z:=x^{-1}(0)\). Since \(x\) is an immersion, \(Z\) is discrete
and hence finite. Write
\[
Z=\{q_1,\ldots,q_k\},
\]
and set
\[
M_i:=\{p\in M:\Phi_t(p)\to q_i\ \text{as}\ t\to-\infty\}.
\]
By the local embedding property of \(x\) near \(q_i\) and the
continuous dependence of the flow on its initial point, each \(M_i\)
is open. The sets \(M_i\) are pairwise disjoint and cover \(M\).
Moreover, \(X(q_i)=0\), so \(q_i\in M_i\), and hence every \(M_i\) is
nonempty. Since \(M\) is connected, we must have \(k=1\).
Denote the unique point of \(Z\) by \(q_0\). Then every backward radial
curve converges to \(q_0\).

Let
\[
\Pi:=dx_{q_0}(T_{q_0}M).
\]
The plane \(\Pi\) is Lagrangian because it is the tangent plane of a
Lagrangian immersion.
Since \(q_0\) is the unique preimage of the origin, compactness of
\(M\) and the local embedding property of \(x\) at \(q_0\) imply that
\(x(M)\) is a smooth embedded submanifold near the origin with tangent
space \(\Pi\). For any \(p\in x(M)\), the radial segment \(sp\),
\(0\leq s\leq1\), lies in \(x(M)\). Hence
\[
p=\left.\frac{d}{ds}\right|_{s=0}sp\in T_0x(M)=\Pi.
\]
Thus \(x(M)\subset\Pi\).

Since \(x(q_0)=0\) and \(x(\partial M)\subset\mathbb S^{2n-1}\), we have
\(q_0\in M^\circ\). Viewing \(x|_{M^\circ}:M^\circ\to\Pi\) as an
immersion between \(n\)-manifolds, the inverse function theorem shows
that \(x(M^\circ)\) is open in \(\Pi\). Since \(M\) is compact,
\(x(M)\) is closed in \(\Pi\), and
\(x(M^\circ)=x(M)\cap(\Pi\cap\mathbb B^{2n})\). Hence \(x(M^\circ)\)
is also closed in \(\Pi\cap\mathbb B^{2n}\). Since it contains the
origin and \(\Pi\cap\mathbb B^{2n}\) is connected, we obtain
\[
	x(M^\circ)=\Pi\cap\mathbb B^{2n}.
\]
Taking closures gives
\(\Pi\cap\overline{\mathbb B}^{2n}\subset x(M)\). Together with
\(x(M)\subset\Pi\cap\overline{\mathbb B}^{2n}\), this yields
\[
	x(M)=\Pi\cap\overline{\mathbb B}^{2n}.
\]
It remains to prove that \(x\) is injective. Suppose that \(x(p)=x(q)\).
Then, for every \(t\leq0\),
\[
x(\Phi_t(p))=e^t x(p)=e^t x(q)=x(\Phi_t(q)).
\]
Both \(\Phi_t(p)\) and \(\Phi_t(q)\) converge to \(q_0\) as
\(t\to-\infty\). Choose a neighborhood \(U\) of \(q_0\) on which \(x\)
is injective. For sufficiently negative \(t\), both points lie in \(U\),
and hence \(\Phi_t(p)=\Phi_t(q)\). Uniqueness of the flow gives \(p=q\).
Thus \(x\) is an injective immersion. Since \(M\) is compact and
\(\mathbb C^n\) is Hausdorff, \(x\) is an embedding. Together with the
image equality above, this proves that \(x\) is a diffeomorphism onto the
equatorial Lagrangian disk \(\Pi\cap\overline{\mathbb B}^{2n}\).
	\end{proof}

\section{Relative rigidity and boundary flux}\label{sec:rigidity}
\subsection{Relative Liouville rigidity and consequences}

We first prove the rigidity theorem under the vanishing of the relative
Liouville class. This assumption gives a global potential for the
Liouville form with zero boundary value, so the proof reduces to the
maximum principle for a drift harmonic function.
	\begin{proof}[\textbf{Proof of Theorem~\ref{thm:relative-liouville-rigidity}}]
Let \(\lambda\) be the Liouville form on \(M\). Since
\(\mathcal L_{\mathrm{rel}}(x)=0\) in
\(H_{\mathrm{dR}}^1(M,\partial M;\mathbb R)\), there exists
\(u\in C^\infty(M)\) such that 
\[
\lambda=du, \qquad u|_{\partial M}=0.
\]
By Lemma~\ref{lem:liouville-self-similar}, with
\(f=\frac{\kappa}{2}|x|^2\), we have \(\delta_f\lambda=0\). Thus
\(\delta du+du(\nabla f)=0\). With the convention
\(\Delta u=-\delta du\), this becomes
\[
\Delta_f u:=\Delta u-\langle\nabla f,\nabla u\rangle=0.
\]
Since \(u=0\) on \(\partial M\), the maximum principle gives
\(u\equiv0\) on \(M\). Consequently, \(\lambda=du\equiv0\). By
Lemma~\ref{lem:liouville-rigidity}, \(x\) is a diffeomorphism onto an
equatorial Lagrangian \(n\)-disk.
\end{proof}
\begin{proof}[\textbf{Proof of Corollary~\ref{thm:exact-self-similar-legendrian}}]
Since \(x\) is exact, there exists \(u\in C^\infty(M)\) such that
	\(\lambda=du\), where \(\lambda\) is the Liouville form of \(x\).
	The Legendrian boundary condition gives
	\(\lambda(Y)=\langle Jx,Y\rangle=0\) for every
	\(Y\in T\partial M\). Hence \(u\) is constant on each connected
	component of \(\partial M\). Since \(\partial M\) is connected, there
	exists \(c\in\mathbb R\) such that \(u=c\) on \(\partial M\).

	Set \(v=u-c\). Then \(v|_{\partial M}=0\) and \(\lambda=dv\). Therefore
	the relative Liouville class vanishes.
	The conclusion now follows from
	Theorem~\ref{thm:relative-liouville-rigidity}.
\end{proof}

\begin{remark}
\label{rem:absolute-relative-exactness}
Suppose that \(\partial M=\Gamma_1\sqcup\cdots\sqcup\Gamma_N\) and
\(\lambda=du\). Since \(\iota^*\lambda=0\), there are constants
\(c_1,\ldots,c_N\) such that \(u|_{\Gamma_j}=c_j\). Then
\[
\mathcal L_{\mathrm{rel}}(x)=0
\quad\Longleftrightarrow\quad
c_1=\cdots=c_N.
\]
Thus absolute exactness need not imply relative exactness when the boundary
is disconnected; the relative Liouville class records the differences
between the boundary values of a Liouville primitive.
\end{remark}

\begin{corollary}
\label{cor:topological-relative-rigidity}
Under the assumptions of Theorem~\ref{thm:relative-liouville-rigidity},
if
\[
H^1_{\mathrm{dR}}(M,\partial M;\mathbb R)=0,
\]
then \(x\) is a diffeomorphism onto an equatorial Lagrangian \(n\)-disk.
\end{corollary}

\begin{proof}
The relative Liouville class vanishes, so the conclusion follows from
Theorem~\ref{thm:relative-liouville-rigidity}.
\end{proof}

\begin{corollary}
\label{cor:homology-relative-rigidity}
Under the assumptions of Theorem~\ref{thm:relative-liouville-rigidity},
suppose in addition that \(M\) is orientable. If
\(H_{n-1}(M;\mathbb R)=0\), then \(x\) is a diffeomorphism onto an
equatorial Lagrangian \(n\)-disk. In particular, when \(n\geq2\), this
holds if \(M\) is a real homology \(n\)-ball.
\end{corollary}

\begin{proof}
By Poincar\'e--Lefschetz duality,
\(H^1_{\mathrm{dR}}(M,\partial M;\mathbb R)\cong
H_{n-1}(M;\mathbb R)\). The result follows from
Corollary~\ref{cor:topological-relative-rigidity}.
\end{proof}

\begin{corollary}
\label{cor:relative-maslov-rigidity}
Under the assumptions of Theorem~\ref{thm:relative-liouville-rigidity},
suppose that \(\kappa\neq0\). Since
\(\iota^*\alpha_H=-\kappa\iota^*\lambda=0\), the mean-curvature form
defines a relative class
\[
\mathcal M_{\mathrm{rel}}(x)
:=[\alpha_H]_{\mathrm{rel}}
\in H^1_{\mathrm{dR}}(M,\partial M;\mathbb R).
\]
If \(\mathcal M_{\mathrm{rel}}(x)=0\), then \(x\) is a diffeomorphism
onto an equatorial Lagrangian \(n\)-disk.
\end{corollary}

\begin{proof}
By Lemma~\ref{lem:liouville-self-similar},
\(\alpha_H=-\kappa\lambda\). Hence
\(\mathcal M_{\mathrm{rel}}(x)=-\kappa\mathcal L_{\mathrm{rel}}(x)\),
and the conclusion follows from
Theorem~\ref{thm:relative-liouville-rigidity}.
\end{proof}

\subsection{Localized boundary unique continuation}

We first fix some notation for differential forms on manifolds with
boundary. Let \(M^n\) be a smooth manifold with boundary \(\partial M\),
and let \(\iota:\partial M\hookrightarrow M\) be the inclusion. For \(\omega\in \Omega^p(M)\), set
\[
	\omega^\top:=\iota^*\omega,
\]
which is the tangential part of \(\omega\) along \(\partial M\). If \(\nu\)
denotes the outward unit normal to \(\partial M\) in \(M\), set
\[
	\omega^\perp:=\iota^*(i_\nu\omega),
\]
which is the normal part of \(\omega\) along \(\partial M\). Thus
\(\omega^\top\in\Omega^p(\partial M)\) and
\(\omega^\perp\in\Omega^{p-1}(\partial M)\).

We use the following notion of vanishing of infinite order \cite{Gerner}.
Let \(\omega\in H^1_{\mathrm{loc}}\Omega^k(M)\) be a locally \(H^1\)
\(k\)-form, and let \(p\in M\) (possibly on \(\partial M\)). We say that
\(\omega\) has a \emph{zero of infinite order in \(1\)-mean} at \(p\) if,
in one (and hence every) chart \(\mu_p\) around \(p\),
\[
\mu_p:U_p\to \mathbb H^n :=
\{(\xi_1,\xi')\in \mathbb R\times\mathbb R^{n-1}
\mid \xi_1\ge 0\},
\]
where \(U_p\) is a neighborhood of \(p\), each local component
\(\omega_{i_1\ldots i_k}\) of \(\omega\) in this coordinate chart satisfies
\[
\lim_{r\to 0^+}
\frac{1}{r^m}
\frac{1}{|B_r(\mu_p(p))\cap \mathbb H^n|}
\int_{B_r(\mu_p(p))\cap \mathbb H^n}
|\omega_{i_1\ldots i_k}|\,d\xi
=0
\qquad\text{for all }m\in \{0,1,2,\ldots\}.
\]
Here \(B_r(\mu_p(p))\) denotes the Euclidean ball in \(\mathbb R^n\)
centered at \(\mu_p(p)\) with radius \(r\), and \(d\xi\) is the Euclidean
volume element in this coordinate chart.
If \(p\) is an interior point, then \(\mathbb H^n\) is replaced by
\(\mathbb R^n\). Since the difference between the averaged and the usual
integral is only a power of \(r\), the averaging may be omitted in the
definition of vanishing of infinite order.

We shall use the following boundary unique continuation theorem for
exterior differential forms due to Gerner. Here \(|\cdot|_g\) denotes the
pointwise norm induced by \(g\).

\begin{theorem}[\cite{Gerner}]
	\label{thm:Gerner-boundary-UCP}
	Let \(M^n\) be an oriented, connected, \(C^{2,1}\)-smooth manifold
	with nonempty boundary, and let \(g\) be a \(C^{1,1}\)-Riemannian
	metric on \(M\). Let \(0\le k\le n\), and suppose that
	\(\omega\in H^1_{\mathrm{loc}}\Omega^k(M)\) satisfies
	\[
	|d\omega|_g^2+|\delta\omega|_g^2
	\le C_K|\omega|_g^2
	\quad\text{a.e.\ on every compact subset }K\subset M,
	\]
	where \(C_K>0\) may depend on \(K\). If \(\omega\) has a zero of
	infinite order in \(1\)-mean at some boundary point
	\(p\in\partial M\), then \(\omega=0\) a.e.\ on \(M\), provided that
	along \(\partial M\) one has either
	\[
	\omega^\top=0
	\qquad\text{or}\qquad
	\omega^\perp=0.
	\]
\end{theorem}
We next derive from Theorem~\ref{thm:Gerner-boundary-UCP} a localized
boundary unique continuation lemma for one-forms satisfying the weighted
harmonic system. This is a key ingredient in the proof of
Theorem~\ref{thm:local-free-boundary-rigidity}. Although
Theorem~\ref{thm:Gerner-boundary-UCP} assumes orientability, the following
lemma removes this assumption by passing to the orientation double cover.
\begin{lemma}
	\label{lem:weighted-uc-forms}
	Let \(M\) be a compact connected smooth Riemannian manifold with
	nonempty boundary, and let \(f\in C^\infty(M)\). Suppose that
	\(\alpha\in\Omega^1(M)\) satisfies
	\[
	d\alpha=0,\qquad \delta_f\alpha=0,
	\qquad \iota^*\alpha=0\quad\text{on }\partial M,
	\]
	where \(\iota:\partial M\hookrightarrow M\) is the inclusion map.
	Assume that there exists a nonempty relatively open subset
	\(\Gamma\subset\partial M\) such that
	\[
	\alpha(\nu)=0\quad\text{on }\Gamma,
	\]
	where \(\nu\) is the outward unit normal along \(\partial M\). Then
	\[
	\alpha\equiv0.
	\]
\end{lemma}

\begin{proof}
	We first prove the assertion under the additional assumption that \(M\) is
	orientable, so that Gerner's boundary unique continuation theorem
	applies.
	
	To verify the required differential inequality, note that, with our
	convention \(\delta_f\alpha=\delta\alpha+\alpha(\nabla f)\), the equation
	\(\delta_f\alpha=0\) gives
	\(\delta\alpha=-\alpha(\nabla f)\). Since \(d\alpha=0\), we have
	\[
	|d\alpha|^2+|\delta\alpha|^2
	=|\alpha(\nabla f)|^2
	\leq |\nabla f|^2|\alpha|^2.
	\]
	As \(M\) is compact, \(|\nabla f|\) is bounded. Hence
	\[
	|d\alpha|^2+|\delta\alpha|^2\leq C|\alpha|^2
	\]
	for some constant \(C>0\). Thus the structural inequality in Gerner's
	theorem is satisfied.
	
	The boundary condition in Gerner's theorem is also satisfied, since
	\(\iota^*\alpha=0\) means that the tangential boundary part of \(\alpha\)
	vanishes on \(\partial M\).
	
	It remains to show that \(\alpha\) has a zero of infinite order in
	\(1\)-mean at some boundary point. Fix \(p\in\Gamma\), and take boundary
	normal coordinates \((y^1,y^2,\ldots,y^n)\) in a neighborhood \(U\) of
	\(p\), with \(U\) chosen sufficiently small that
	\(\partial M\cap U\subset\Gamma\), and with \(\mu_p(p)=0\), such that
	\begin{equation}\label{boundary-normal-coordinates}
		M\cap U=\{y^1\geq0\},\qquad
		\partial M\cap U=\{y^1=0\},\qquad
		g=(dy^1)^2+g_{ab}(y)\,dy^a dy^b.
	\end{equation}
	Throughout this proof we adopt the following index convention:
	\[
	1\leq A,B,C\leq n,\qquad 2\leq a,b,c\leq n.
	\]
	Write \(\alpha=u\,dy^1+v_a\,dy^a\). Since
	\(\iota^*\alpha=0\) on \(\partial M\), we have \(v_a=0\) on
	\(\partial M\cap U\). Moreover, since
	\(\alpha(\nu)=0\) on \(\Gamma\), we also have \(u=0\) on
	\(\partial M\cap U\). Hence all tangential derivatives of \(u\) and
	\(v_a\) vanish on \(\partial M\cap U\).

	The equation \(d\alpha=0\) gives
	\(\partial_1v_a=\partial_a u\). Thus
	\[
	\partial_1v_a=0
	\qquad\text{on }\partial M\cap U.
	\]

	Next we use \(\delta_f\alpha=0\) to obtain the normal derivative of \(u\).
	Let
	\[
	V=\alpha^\sharp=u\partial_1+g^{ab}v_b\partial_a.
	\]
	Since \(\delta\alpha=-\operatorname{div}V\), the equation
	\(\delta_f\alpha=0\) is equivalent to
	\[
	\operatorname{div}V=\langle V,\nabla f\rangle.
	\]
	By \eqref{boundary-normal-coordinates}, we have
	\(\Gamma^1_{11}=\Gamma^1_{1a}=0\), where
	\(\nabla_{\partial_A}\partial_B=\Gamma^C_{AB}\partial_C\).
	Hence, by the definition of divergence, we obtain
	\[
	\operatorname{div}V
	=
	\partial_1u+\partial_a(g^{ab}v_b)
	+\Gamma^a_{a1}u+\Gamma^a_{ab}g^{bc}v_c.
	\]
	On the other hand,
	\[
	\langle V,\nabla f\rangle
	=
	u\,\partial_1f+g^{ab}v_b\,\partial_af.
	\]
	Therefore
	\begin{equation}\label{normal-derivative-u}
		\partial_1u
		=
		-\partial_a(g^{ab}v_b)
		-\Gamma^a_{a1}u
		-\Gamma^a_{ab}g^{bc}v_c
		+u\,\partial_1f
		+g^{ab}v_b\,\partial_af.
	\end{equation}
	The right-hand side vanishes on \(\partial M\cap U\), since
	\(u=0\), \(v_a=0\), and \(\partial_av_b=0\) there. Hence
	\(\partial_1u=0\) on \(\partial M\cap U\).

	We now prove by induction that, for every \(k\geq0\),
	\[
	\partial_1^ku=0,\qquad
	\partial_1^kv_a=0
	\qquad\text{on }\partial M\cap U.
	\]
	The cases \(k=0\) and \(k=1\) have already been proved. Suppose the claim
	holds up to order \(k\). Then all tangential derivatives of these
	identities also vanish on \(\partial M\cap U\). Applying
	\(\partial_1^k\) to \(\partial_1v_a=\partial_au\), we get
	\[
	\partial_1^{k+1}v_a
	=
	\partial_a\partial_1^ku
	=
	0
	\qquad\text{on }\partial M\cap U.
	\]
	Applying \(\partial_1^k\) to \eqref{normal-derivative-u}, every term on the
	right-hand side contains a factor of the form
	\[
	\partial_1^\ell u,\qquad
	\partial_1^\ell v_a,\qquad
	\partial_a\partial_1^\ell v_b,
	\qquad \ell\leq k.
	\]
	By the induction hypothesis and its tangential derivatives, all these
	factors vanish on \(\partial M\cap U\). Therefore
	\[
	\partial_1^{k+1}u=0
	\qquad\text{on }\partial M\cap U.
	\]
	This proves the induction.

	Thus \(\partial_1^ku=0\) and \(\partial_1^kv_a=0\) on
	\(\partial M\cap U\) for every \(k\geq0\). Taking tangential derivatives
	gives
	\[
	\partial_{\hat y}^I\partial_1^ku=0,\qquad
	\partial_{\hat y}^I\partial_1^kv_a=0
	\qquad\text{on }\partial M\cap U
	\]
	for every tangential multi-index \(I\), where
	\(\hat y=(y^2,\ldots,y^n)\). Hence all partial derivatives of all local
	components of \(\alpha\) vanish at \(p\).

	Let \(\alpha_A\) denote a local component of \(\alpha\), so that
	\(\alpha_1=u\) and \(\alpha_a=v_a\). Since all partial derivatives of
	\(\alpha_A\) vanish at \(p\), Taylor's formula gives that, for every
	\(N\geq1\), there exists \(C_N>0\) such that
	\[
	|\alpha_A(y)|\leq C_N|y|^N
	\]
	in a coordinate half-ball around \(0=\mu_p(p)\). Since \(|y|\leq r\) on
	\(B_r(0)\cap\mathbb H^n\), it follows that
	\[
	|\alpha_A(y)|\leq C_Nr^N
	\qquad\text{on }B_r(0)\cap\mathbb H^n.
	\]
	Therefore, for all sufficiently small \(r>0\),
	\[
	\frac{1}{|B_r(0)\cap\mathbb H^n|}
	\int_{B_r(0)\cap\mathbb H^n}|\alpha_A(y)|\,dy
	\leq C_Nr^N.
	\]
	Given any \(m\geq0\), choosing \(N>m\) gives
	\[
	\frac{1}{r^m}
	\frac{1}{|B_r(0)\cap\mathbb H^n|}
	\int_{B_r(0)\cap\mathbb H^n}|\alpha_A(y)|\,dy
	\leq C_Nr^{N-m}\longrightarrow0
	\]
	as \(r\to0\). Thus \(\alpha\) vanishes to infinite order in
	\(1\)-mean at \(p\). Hence Gerner's theorem applies and yields
	\(\alpha\equiv0\).

	To remove the orientability assumption, let \(M\) be non-orientable, and
	denote by \(\pi:\widehat M\to M\) its orientation double cover. Then
	\(\widehat M\) is a compact, connected, orientable manifold with boundary
	\(\partial\widehat M=\pi^{-1}(\partial M)\). Equip \(\widehat M\) with the
	pullback metric \(\widehat g=\pi^*g\), and set
	\[
	\widehat f=\pi^*f,\qquad
	\widehat\alpha=\pi^*\alpha,\qquad
	\widehat\Gamma=\pi^{-1}(\Gamma).
	\]
	Since \(\pi:(\widehat M,\widehat g)\to(M,g)\) is a local isometry, both
	the exterior derivative and the codifferential commute with pullback. It
	follows that
	\[
	d\widehat\alpha=0,\qquad
	\widehat\delta_{\widehat f}\widehat\alpha=0.
	\]
	Moreover,
	\[
	\widehat\iota^*\widehat\alpha=0
	\quad\text{on }\partial\widehat M,
	\qquad
	\widehat\alpha(\widehat\nu)=0
	\quad\text{on }\widehat\Gamma.
	\]
	Applying the orientable case to
	\((\widehat M,\widehat g)\), we obtain
	\(\widehat\alpha\equiv0\). Since \(\pi\) is surjective, this implies
	\(\alpha\equiv0\) on \(M\).
\end{proof}
\begin{remark}
	When \(f\) is constant, \(\delta_f=\delta\), so \(\alpha\) is a
	harmonic one-form:
	\[
	\Delta_H\alpha=(d\delta+\delta d)\alpha=0.
	\]
	Under the stronger assumption that \(\alpha=0\) on the whole boundary,
	the conclusion also follows from the boundary unique continuation result
	for harmonic forms in \cite{CDGM}. Here, however, the normal component of
	\(\alpha\) is required to vanish only on a nonempty relatively open
	subset of the boundary.
\end{remark}
\subsection{Boundary rigidity and flux}

We now apply the localized boundary unique continuation lemma to the Liouville form.
We first prove the local free-boundary rigidity theorem, then derive the boundary Liouville flux identity and its rigidity consequences.

\begin{proof}[\textbf{Proof of Theorem~\ref{thm:local-free-boundary-rigidity}}]
	Let \(\lambda\) be the Liouville form on \(M\). By
	Lemma~\ref{lem:liouville-self-similar}, we have \(d\lambda=0\) and
	\(\delta_f\lambda=0\), where \(f=\frac{\kappa}{2}|x|^2\).

	The Legendrian condition gives \(\iota^*\lambda=0\) on \(\partial M\).
	  Using \(\nu=\sin\theta x +\cos\theta Jx\) and \(|x|=1\) along \(\partial M\), we obtain  
    \[
    \lambda(\nu)=\langle Jx, \nu\rangle =\cos\theta=0 \quad \text{on} \quad \Gamma.
    \] 
    Hence Lemma~\ref{lem:weighted-uc-forms} yields \(\lambda\equiv0\).
	The result now follows from Lemma~\ref{lem:liouville-rigidity}.
\end{proof}

\begin{proof}[\textbf{Proof of Theorem~\ref{thm:boundary-liouville-flux}}]
	Let \(W:=\lambda^\sharp\) be the metric dual vector field of
	the Liouville form \(\lambda\). 
    By Lemma~\ref{lem:liouville-self-similar},
	\[\delta_f\lambda=0, \qquad f=\frac{\kappa}{2}|x|^2.\] 
    Then \(\delta_f\lambda=0\) is equivalent to
	\[
	\operatorname{div}_fW
	:=
	\operatorname{div}W-\langle\nabla f,W\rangle=0.
	\]
	The weighted divergence theorem gives
	\[
	0
	=
	\int_M\operatorname{div}_fW\,e^{-f}\,d\mu
	=\int_{\partial M}\langle W,\nu\rangle e^{-f}\,d\sigma
		=
	\int_{\partial M}\lambda(\nu)e^{-f}\,d\sigma.
	\]
	Here the divergence theorem is understood for Riemannian densities, so
	no orientability assumption is required. 
    
    Since
	\(\lambda(\nu)=\cos\theta\) and \(|x|=1\) on \(\partial M\),
	we have \(f=\kappa/2\) there. Therefore
	\[
	0
	=
	e^{-\kappa/2}\int_{\partial M}\cos\theta\,d\sigma,
	\]
	which proves the first assertion.

	If the contact angle is constant, equal to \(\theta_j\), on each
	\(\Gamma_j\), then
	\[
	0
	=
	\int_{\partial M}\cos\theta\,d\sigma
    =\sum_{j=1}^N \int_{\Gamma_j} \cos\theta_j d\sigma 
	=
	\sum_{j=1}^N|\Gamma_j|\cos\theta_j.
	\]
    This completes the proof.
\end{proof}

\begin{corollary}
	\label{cor:fixed-sign-rigidity}
	Let \(x:M^n\to\overline{\mathbb B}^{2n}\subset\mathbb C^n\) be a
	compact Lagrangian self-similar immersion with Legendrian
	boundary on \(\mathbb S^{2n-1}\). Let \(\theta\in[0,\pi)\) denote the contact angle.
   If \(\cos\theta\) has a fixed sign
	on \(\partial M\), then \(x\) is a diffeomorphism onto an equatorial
	Lagrangian \(n\)-disk.

	In particular, the same conclusion holds if \(\partial M\) is
	connected and \(x\) has Legendrian capillary boundary.
\end{corollary}

\begin{proof}
	By Theorem~\ref{thm:boundary-liouville-flux},
	\[
	\int_{\partial M}\cos\theta\,d\sigma=0.
	\]
	Since \(\cos\theta\) has a fixed sign, it follows that
	\(\cos\theta\equiv0\) on \(\partial M\). As
	\(\theta\in[0,\pi)\), we have \(\theta=\pi/2\) on \(\partial M\).
	The conclusion now follows from
	Theorem~\ref{thm:local-free-boundary-rigidity}.

	If \(\partial M\) is connected and the boundary is capillary, then
	\(\theta\) is constant on \(\partial M\), so \(\cos\theta\) has a
	fixed sign. The last assertion follows from the first one.
\end{proof}

\section{Examples of capillary Lagrangian submanifolds}
\label{sec:examples}

In this section, we construct non-disk-type Lagrangian self-similar
submanifolds with Legendrian capillary boundary in the unit ball. We first
recall Anciaux's construction in the form needed below, and then use it to
produce a family of compact embedded examples with two boundary components.

\subsection{Anciaux's construction}

Let
\(\psi:\Sigma^{n-1}\to\mathbb S^{2n-1}\) be a Legendrian immersion, and
let \(\gamma:I\to\mathbb C^*\) be a regular plane curve parametrized by
arclength \(s\). Define
\(F:I\times\Sigma\to\mathbb C^n\) by
\(F(s,p)=\gamma(s)\psi(p)\). As observed by Ros and
Urbano~\cite{RosUrbano} and Anciaux~\cite{Anciaux}, the Legendrian
condition on \(\psi\) implies that \(F\) is Lagrangian.

Let \(N=J\gamma'\) be the unit normal vector field of \(\gamma\), and let
\(k\) denote its curvature. The following formulas are standard in this
construction; see Ros and Urbano~\cite{RosUrbano} and
Anciaux~\cite{Anciaux}. The normal projection of the position vector is
\begin{equation}\label{normalpro}
	F^\perp
	=
	\langle\gamma,N\rangle\,J\gamma'\cdot\psi,
\end{equation}
and the mean curvature vector of \(F\) is
\begin{equation}\label{meancurv_general}
	H
	=
	\left(
		k-(n-1)\frac{\langle\gamma,N\rangle}{|\gamma|^2}
	\right)
	J\gamma'\cdot\psi
	+
	\frac{1}{|\gamma|^2}\gamma\cdot H_\psi,
\end{equation}
where \(H_\psi\) is the mean curvature vector of \(\psi\) in
\(\mathbb S^{2n-1}\).

If \(\psi\) is minimal, then \(H_\psi=0\), and hence
\begin{equation}\label{meancurv}
	H
	=
	\left(
		k-(n-1)\frac{\langle\gamma,N\rangle}{|\gamma|^2}
	\right)
	J\gamma'\cdot\psi.
\end{equation}
Combining \eqref{normalpro} and \eqref{meancurv}, we see that, for any
\(\kappa\in\mathbb R\), the immersion \(F\) satisfies
\(H+\kappa F^\perp=0\) if and only if
\begin{equation}\label{eq:profile-self-similar}
	k
	=
	\langle\gamma,N\rangle
	\left(
		\frac{n-1}{|\gamma|^2}-\kappa
	\right).
\end{equation}
This is the profile-curve equation appearing in Anciaux's construction
\cite{Anciaux}.

Write \(\gamma(s)=r(s)e^{i\varphi(s)}\) and
\(\gamma'(s)=e^{i\alpha(s)}\), and set
\(\Theta:=\alpha-\varphi\). Since \(k=\alpha'\) and
\(\langle\gamma,N\rangle=-r\sin\Theta\), equation
\eqref{eq:profile-self-similar}, together with the arclength
parametrization, yields, as in Anciaux~\cite{Anciaux},
\begin{equation}\label{eq:unified-Anciaux-system}
	\begin{cases}
		r'=\cos\Theta,\\[2mm]
		\displaystyle \varphi'=\frac{\sin\Theta}{r},\\[3mm]
		\displaystyle
		\Theta'
		=
		\left(\kappa r-\frac nr\right)\sin\Theta.
	\end{cases}
\end{equation}
For \(\kappa=1,0,-1\), this gives respectively the shrinking, minimal,
and expanding cases.

\subsection{Embedded two-boundary self-similar examples}

We now use the preceding construction to obtain a one-parameter family of
compact embedded exact Lagrangian self-similar submanifolds with two
Legendrian capillary boundary components.

\begin{theorem}
	\label{thm:embedded-self-similar-pieces}
	Let \(n\geq2\), \(\kappa\in\{-1,0,1\}\), and \(a\in(0,1)\).
	Then there exist \(s_a>0\) and a proper embedding
	\[
		F_{a,\kappa}:M_a=[-s_a,s_a]\times\mathbb S^{n-1}
		\longrightarrow\overline{\mathbb B}^{2n}
	\]
	whose image is a compact exact Lagrangian submanifold satisfying
	\(H+\kappa F_{a,\kappa}^{\perp}=0\). Its two boundary components are
	totally geodesic Legendrian spheres. Define
	\begin{equation}\label{eq:qak}
		q_{a,\kappa}
		:=
		a^n\exp\left(\frac{\kappa}{2}(1-a^2)\right).
	\end{equation}
	Then \(0<q_{a,\kappa}<1\), and, with the convention
	\(\nu=\sin\theta\,F+\cos\theta\,JF\), their contact angles satisfy
	\begin{equation}\label{eq:contact-unified}
		\cos\theta_+=q_{a,\kappa},
		\qquad
		\cos\theta_-=-q_{a,\kappa}.
	\end{equation}
	In particular, \(\theta_++\theta_-=\pi\), and neither boundary
	component is free. Moreover,
	\[
		\mathcal L_{\mathrm{rel}}(F_{a,\kappa})\neq0
		\quad\text{in }H^1_{\mathrm{dR}}(M_a,\partial M_a;\mathbb R).
	\]
\end{theorem}

The embedding in Theorem~\ref{thm:embedded-self-similar-pieces} is given
explicitly as follows. Let \(r,\varphi,\Theta\) be the solution of
\eqref{eq:unified-Anciaux-system} with initial conditions
\begin{equation}\label{eq:unified-initial-data}
	r(0)=a,\qquad
	\varphi(0)=0,\qquad
	\Theta(0)=\frac{\pi}{2}.
\end{equation}
The proof shows that the solution is defined on \([-s_a,s_a]\), with
\(r(\pm s_a)=1\) and \(a\leq r(s)<1\) for \(|s|<s_a\). Let
\[
	\psi_{n-1}:\mathbb S^{n-1}\longrightarrow\mathbb S^{2n-1},
	\qquad
	\psi_{n-1}(p)=p,
\]
be the standard totally geodesic Legendrian embedding, where
\(\mathbb S^{n-1}\subset\mathbb R^n\subset\mathbb C^n\), and set
\(\gamma_{a,\kappa}(s)=r(s)e^{i\varphi(s)}\). Then
\begin{equation}\label{eq:unified-example}
	F_{a,\kappa}(s,p)
	=
	\gamma_{a,\kappa}(s)\psi_{n-1}(p)
	=
	\gamma_{a,\kappa}(s)p.
\end{equation}
Its two boundary components are
\[
	\Gamma_\pm
	=
	e^{\pm i\Phi_{a,\kappa}}\mathbb S^{n-1}
	\subset\mathbb S^{2n-1},
\]
where \(\Phi_{a,\kappa}:=\varphi(s_a)\in(0,\pi/2)\). Equivalently,
\(\theta_+=\arccos q_{a,\kappa}\) and
\(\theta_-=\pi-\arccos q_{a,\kappa}\).

\begin{proof}
	For brevity, write \(F=F_{a,\kappa}\) and
	\(\gamma=\gamma_{a,\kappa}\).

	\medskip
	\noindent
	\textbf{The radial segment.}
	Set \(h_\kappa(r):=r^n e^{-\kappa r^2/2}\). For \(0<r\leq1\),
	\begin{equation}\label{eq:h-monotone}
		h_\kappa'(r)
		=
		r^{n-1}e^{-\kappa r^2/2}
		\left(n-\kappa r^2\right)>0.
	\end{equation}
	Thus \(h_\kappa\) is strictly increasing on \((0,1]\).

	The system \eqref{eq:unified-Anciaux-system} has the first integral
	\begin{equation}\label{eq:first-integral-general}
		h_\kappa(r)\sin\Theta
		=
		r^n e^{-\kappa r^2/2}\sin\Theta.
	\end{equation}
	Indeed,
	\[
		\begin{aligned}
			\frac{d}{ds}\bigl(h_\kappa(r)\sin\Theta\bigr)
			&=
			h_\kappa(r)
			\left[
				\left(\frac nr-\kappa r\right)r'\sin\Theta
				+\cos\Theta\,\Theta'
			\right] \\
			&=0.
		\end{aligned}
	\]
	By \eqref{eq:unified-initial-data},
	\begin{equation}\label{eq:sin-theta-first}
		\sin\Theta(s)
		=
		\frac{h_\kappa(a)}{h_\kappa(r(s))}.
	\end{equation}

	At \(s=0\), one has \(\Theta'(0)=\kappa a-n/a<0\) and
	\(r''(0)=n/a-\kappa a>0\). Hence \(r\) has a strict local minimum at
	\(s=0\). On the right-hand branch, as long as \(r\leq1\), we have
	\(\sin\Theta>0\) and \(\Theta'<0\). Therefore
	\(0<\Theta<\pi/2\) and \(r'>0\).

	From \eqref{eq:sin-theta-first},
	\begin{equation}\label{eq:r-separation}
		r'
		=
		\sqrt{
			1-
			\left(
				\frac{h_\kappa(a)}{h_\kappa(r)}
			\right)^2
		}.
	\end{equation}
	The time required for \(r\) to increase from \(a\) to \(1\) is
	\begin{equation}\label{eq:sa-integral}
		s_a
		=
		\int_a^1
		\frac{d\rho}{
			\sqrt{
				1-
				\left(
					h_\kappa(a)/h_\kappa(\rho)
				\right)^2
			}
		}.
	\end{equation}
	As \(\rho\to a^+\),
	\[
		1-
		\left(
			\frac{h_\kappa(a)}{h_\kappa(\rho)}
		\right)^2
		=
		2\frac{h_\kappa'(a)}{h_\kappa(a)}(\rho-a)
		+
		O\bigl((\rho-a)^2\bigr),
	\]
	so the integral is finite.

	The transformation \(\widetilde r(s)=r(-s)\),
	\(\widetilde\varphi(s)=-\varphi(-s)\), and
	\(\widetilde\Theta(s)=\pi-\Theta(-s)\) preserves
	\eqref{eq:unified-Anciaux-system} and
	\eqref{eq:unified-initial-data}. Uniqueness therefore gives
	\begin{equation}\label{eq:symmetry-solutions}
		r(-s)=r(s),\qquad
		\varphi(-s)=-\varphi(s),\qquad
		\Theta(-s)=\pi-\Theta(s).
	\end{equation}
	Consequently,
	\begin{equation}\label{eq:property_r}
		r(\pm s_a)=1
		\quad\text{and}\quad
		a\leq r(s)<1
		\quad\text{for }|s|<s_a.
	\end{equation}

	By \eqref{eq:qak}, \(q_{a,\kappa}=h_\kappa(a)/h_\kappa(1)\). Since
	\(h_\kappa\) is strictly increasing, \(0<q_{a,\kappa}<1\). Set
	\(\beta_{a,\kappa}:=\arcsin q_{a,\kappa}\in(0,\pi/2)\). Then
	\begin{equation}\label{eq:endpoint-Theta}
		\Theta(s_a)=\beta_{a,\kappa},
		\qquad
		\Theta(-s_a)=\pi-\beta_{a,\kappa}.
	\end{equation}

	\medskip
	\noindent
	\textbf{The phase variation.}
	On \((0,s_a)\), the function \(\Theta\) is strictly decreasing and
	may be used as a parameter. From \eqref{eq:unified-Anciaux-system},
	\(d\varphi/d\Theta=-1/(n-\kappa r^2)\). Therefore
	\begin{equation}\label{eq:Phi-integral}
		\Phi_{a,\kappa}:=\varphi(s_a)
		=
		\int_{\beta_{a,\kappa}}^{\pi/2}
		\frac{d\vartheta}{n-\kappa r(\vartheta)^2}.
	\end{equation}
	Since \(n-\kappa r^2\geq n-1\geq1\), we obtain
	\[
		0<\Phi_{a,\kappa}
		\leq
		\frac{\frac{\pi}{2}-\beta_{a,\kappa}}{n-1}
		<
		\frac{\pi}{2}.
	\]
	In particular,
	\begin{equation}\label{eq:phase-less-pi}
		2\Phi_{a,\kappa}<\pi.
	\end{equation}
	Moreover, \(\varphi'=\sin\Theta/r>0\), so \(\varphi\) is strictly
	increasing and \(\varphi([-s_a,s_a])=[-\Phi_{a,\kappa},\Phi_{a,\kappa}]\).

	\medskip
	\noindent
	\textbf{The self-similar immersion.}
	Set \(\alpha:=\varphi+\Theta\). The first two equations in
	\eqref{eq:unified-Anciaux-system} give
	\(\gamma'=e^{i\varphi}(r'+ir\varphi')=e^{i\alpha}\). Thus \(\gamma\)
	is parametrized by arclength and \(\Theta=\alpha-\varphi\). Since
	\(\psi_{n-1}\) is a minimal Legendrian immersion, the construction in
	the preceding subsection shows that \(F\) is Lagrangian and satisfies
	\(H+\kappa F^\perp=0\).

	The induced metric is \(g=ds^2+r(s)^2g_{\mathbb S^{n-1}}\). Since
	\(r(s)\geq a>0\), the map \(F\) is an immersion of \(M_a\).

	\medskip
	\noindent
	\textbf{Embeddedness.}
	Suppose that \(\gamma(s)=\gamma(t)\). Then
	\(\varphi(s)-\varphi(t)\in2\pi\mathbb Z\). By
	\eqref{eq:phase-less-pi},
	\(|\varphi(s)-\varphi(t)|\leq2\Phi_{a,\kappa}<\pi\), forcing
	\(\varphi(s)=\varphi(t)\). Since \(\varphi\) is strictly increasing,
	\(s=t\). Similarly, \(\gamma(s)=-\gamma(t)\) is impossible because it
	would imply \(\varphi(s)-\varphi(t)\in\pi+2\pi\mathbb Z\).

	Now suppose that \(F(s,p)=F(t,q)\). Since
	\(p,q\in\mathbb S^{n-1}\) are real unit vectors, the quotient
	\(\gamma(s)/\gamma(t)\) is real and has absolute value one. Hence it
	equals \(1\) or \(-1\). The second possibility is impossible, while
	the first gives \(s=t\) and \(p=q\). Thus \(F\) is injective. Since
	\(M_a\) is compact, \(F\) is an embedding.

	\medskip
	\noindent
	\textbf{Boundary geometry and contact angles.}
	By \eqref{eq:symmetry-solutions},
	\(\varphi(\pm s_a)=\pm\Phi_{a,\kappa}\), and hence
	\[
		\Gamma_\pm=e^{\pm i\Phi_{a,\kappa}}\mathbb S^{n-1}.
	\]
	Each \(\Gamma_\pm\) is the unit sphere in the real Lagrangian
	\(n\)-plane \(e^{\pm i\Phi_{a,\kappa}}\mathbb R^n\), and is therefore
	a totally geodesic Legendrian sphere.

	Since \(\gamma\) is parametrized by arclength and \(F_s\) is orthogonal
	to \(F_*(T\mathbb S^{n-1})\), the outward unit conormals along
	\(s=s_a\) and \(s=-s_a\) are \(F_s\) and \(-F_s\), respectively.
	Moreover, since \(r=1\) on the boundary,
	\[
		F_s=e^{i(\varphi+\Theta)}p
		=
		\cos\Theta\,F+\sin\Theta\,JF.
	\]
	Along \(s=s_a\), using \(\Theta(s_a)=\beta_{a,\kappa}\), we obtain
	\[
		\nu_+
		=
		\cos\beta_{a,\kappa}\,F
		+
		\sin\beta_{a,\kappa}\,JF.
	\]
	Thus \(\cos\theta_+=\sin\beta_{a,\kappa}=q_{a,\kappa}\).

	Similarly, along \(s=-s_a\), using
	\(\Theta(-s_a)=\pi-\beta_{a,\kappa}\), we obtain
	\[
		\nu_-
		=
		\cos\beta_{a,\kappa}\,F
		-
		\sin\beta_{a,\kappa}\,JF.
	\]
	Hence \(\cos\theta_-=-\sin\beta_{a,\kappa}=-q_{a,\kappa}\). This
	proves \eqref{eq:contact-unified}.

	Since \(|F(s,p)|=r(s)\), relation \eqref{eq:property_r} implies
	\(F^{-1}(\mathbb S^{2n-1})=\partial M_a\) and
	\(F(M_a^\circ)\subset\mathbb B^{2n}\). Moreover, for
	\(\rho(z)=|z|^2-1\), one has \(d(\rho\circ F)(\partial_s)=2rr'\),
	which is nonzero at \(s=\pm s_a\). Hence
	\(F\pitchfork\mathbb S^{2n-1}\). Since \(F\) is already a proper
	embedding, this shows that
	\(F:(M_a,\partial M_a)\to
	(\overline{\mathbb B}^{2n},\mathbb S^{2n-1})\) is a neat embedding.

	\medskip
	\noindent
	\textbf{Exactness and the relative Liouville class.}
	Let \(\lambda\) be the pulled-back Liouville form. For
	\(v\in T_p\mathbb S^{n-1}\), one has \(\lambda(v)=0\), while
	\(\lambda(\partial_s)=\langle JF,F_s\rangle=r^2\varphi'=r\sin\Theta\).
	Hence
	\begin{equation}\label{eq:lambda-unified-example}
		\lambda=r(s)\sin\Theta(s)\,ds.
	\end{equation}
	Define \(G(s):=\int_{-s_a}^{s}r(\tau)\sin\Theta(\tau)\,d\tau\).
	Then \(\lambda=dG\). However, \(G|_{\Gamma_-}=0\), whereas
	\[
		G|_{\Gamma_+}
		=
		\int_{-s_a}^{s_a}r(\tau)\sin\Theta(\tau)\,d\tau
		>0.
	\]
	Thus no primitive of \(\lambda\) takes the same constant value on
	both boundary components, and consequently
	\[
		\mathcal L_{\mathrm{rel}}(F_{a,\kappa})\neq0
		\quad\text{in }H^1_{\mathrm{dR}}(M_a,\partial M_a;\mathbb R).
	\]
\end{proof}

\subsection{The minimal case and Lagrangian catenoids}
	When \(\kappa=0\), the family in
	Theorem~\ref{thm:embedded-self-similar-pieces} recovers, up to a
	unitary rotation, compact portions of the Lagrangian catenoids of
	Harvey--Lawson~\cite{HarveyLawson} and
	Castro--Urbano~\cite{CastroUrbano}. Recall that, for \(0<c<1\),
	\[
		M_c^n
		=
		\left\{
		(x,y)\in\mathbb C^n\cong\mathbb R^n\times\mathbb R^n
		\;\middle|\;
		|x|y=|y|x,\ 
		\operatorname{Im}\bigl((|x|+i|y|)^n\bigr)=c,\ 
		0<\arg(|x|+i|y|)<\frac{\pi}{n}
		\right\}.
	\]

	We first recall the relation between \(M_c^n\) and Anciaux's
	construction. Let \(\gamma=\rho e^{i\phi}\) be an
	arclength-parametrized generating curve and write
	\(\gamma'=e^{i\eta}\). In the minimal case, the generating-curve
	equations admit the first integrals
	\[
		\rho^n\sin(\phi-\eta)=c_1,
		\qquad
		\eta+(n-1)\phi=c_2.
	\]
	After a constant unitary rotation, we may assume \(c_2=0\).
	Taking \(c_1=c>0\), we obtain
	\(\eta=(1-n)\phi\) and \(\rho^n\sin(n\phi)=c\).

	Choose the branch
	\[
		L_c
		:=
		\left\{
		\rho e^{i\phi}\in\mathbb C^*
		\;\middle|\;
		\rho^n\sin(n\phi)=c,\quad
		0<\phi<\frac{\pi}{n}
		\right\}.
	\]
	It is parametrized by
	\[
		\beta(\phi)
		=
		\left(\frac{c}{\sin(n\phi)}\right)^{1/n}e^{i\phi},
		\qquad
		0<\phi<\frac{\pi}{n}.
	\]
	A direct computation gives
	\[
		\beta'(\phi)
		=
		-c^{1/n}\bigl(\sin(n\phi)\bigr)^{-1-1/n}
		e^{i(1-n)\phi}.
	\]
	Thus \(L_c\) is regular, and both of its ends have infinite length,
	since \(|\beta'(\phi)|\) has order \(\phi^{-1-1/n}\) as
	\(\phi\to0^+\) and order
	\((\pi/n-\phi)^{-1-1/n}\) as \(\phi\to(\pi/n)^-\).
	Consequently, \(L_c\) has a global arclength parametrization
	\[
		\gamma:\mathbb R\longrightarrow L_c.
	\]
	Orienting it by decreasing \(\phi\), we have
	\[
		\gamma'(s)=e^{i(1-n)\phi(s)},
	\]
	so that \(\eta=(1-n)\phi\), in agreement with the choice \(c_2=0\),
	and \(\gamma(\mathbb R)=L_c\).

	Define \(F^c(s,p)=\gamma(s)p\) for
	\(p\in\mathbb S^{n-1}\). We claim that
	\[
		F^c(\mathbb R\times\mathbb S^{n-1})=M_c^n.
	\]
	Indeed, writing
	\[
		F^c(s,p)
		=
		\rho e^{i\phi}p
		=
		\bigl(\rho\cos\phi\,p,\rho\sin\phi\,p\bigr)
		=:(x,y),
	\]
	we have \(|x|y=|y|x\) and
	\[
		\operatorname{Im}\bigl((|x|+i|y|)^n\bigr)
		=
		\rho^n\sin(n\phi)
		=
		c.
	\]
	Moreover, \(\arg(|x|+i|y|)=\phi\in(0,\pi/n)\). Hence the image of
\(F^c\) is contained in \(M_c^n\).

	Conversely, let \(z=(x,y)\in M_c^n\). The angle condition
	\(0<\arg(|x|+i|y|)<\pi/n\) rules out \(x=0\), since otherwise
	\(\arg(|x|+i|y|)=\pi/2\geq\pi/n\). Thus \(x\neq0\), and the
	relation \(|x|y=|y|x\) implies that \(y\) is a
	nonnegative multiple of \(x\). Set
	\[
		p:=\frac{x}{|x|},
		\qquad
		\rho:=\sqrt{|x|^2+|y|^2},
		\qquad
		\phi:=\arg(|x|+i|y|)
	=\arctan\frac{|y|}{|x|}.
	\]
	Then \(p\in\mathbb S^{n-1}\), \(0<\phi<\pi/n\), and
	\(z=\rho e^{i\phi}p\). The defining equation of \(M_c^n\) gives
	\[
		\rho^n\sin(n\phi)
		=
		\operatorname{Im}\bigl((|x|+i|y|)^n\bigr)
		=
		c.
	\]
	Thus
	\[
		\rho e^{i\phi}\in L_c=\gamma(\mathbb R).
	\]
	Hence there exists \(s\in\mathbb R\) such that
	\(\gamma(s)=\rho e^{i\phi}\), and therefore
	\[
		z=\rho e^{i\phi}p=\gamma(s)p=F^c(s,p).
	\]
	Thus \(z\) belongs to the image of \(F^c\). Therefore
	\[
		F^c(\mathbb R\times\mathbb S^{n-1})=M_c^n.
	\]

	We now compare this generating curve with the one used in
	Theorem~\ref{thm:embedded-self-similar-pieces}. Set \(c=a^n\), and
	write
	\[
		\gamma_{a,0}(s)=r(s)e^{i\varphi(s)},
		\qquad
		\gamma_{a,0}'(s)=e^{i\alpha(s)},
		\qquad
		\alpha=\varphi+\Theta.
	\]
	For \(\kappa=0\), the first integral and the initial conditions give
	\[
		r^n\sin\Theta=c,
		\qquad
		\alpha+(n-1)\varphi=\frac{\pi}{2}.
	\]
	Define
	\[
		\widetilde\gamma(s)
		:=
		e^{i\pi/(2n)}\gamma_{a,0}(-s).
	\]
	Write
	\(\widetilde\gamma=\widetilde r e^{i\widetilde\varphi}\) and
	\(\widetilde\gamma'=e^{i\widetilde\alpha}\). Modulo \(2\pi\), we have
	\[
		\widetilde\varphi(s)
		=
		\varphi(-s)+\frac{\pi}{2n},
		\qquad
		\widetilde\alpha(s)
		=
		\alpha(-s)+\pi+\frac{\pi}{2n}.
	\]
	Consequently,
	\[
		\widetilde\alpha+(n-1)\widetilde\varphi
		\equiv0\pmod{2\pi},
	\]
	and
	\[
		\widetilde r^{\,n}
		\sin(\widetilde\varphi-\widetilde\alpha)
		=
		r(-s)^n\sin\Theta(-s)
		=
		c.
	\]
	It follows that
	\(\widetilde r^{\,n}\sin(n\widetilde\varphi)=c\). Thus
	\(\widetilde\gamma([-s_a,s_a])\) is a segment of the generating
	curve \(L_c\) used in the Anciaux parametrization of \(M_c^n\).
	In particular, it is obtained from the profile curve in
	Theorem~\ref{thm:embedded-self-similar-pieces} by reversing the
	arclength parameter and applying the constant unitary rotation
	\(e^{i\pi/(2n)}\).

	Let \(\beta:=\arcsin c\). Then
	\(\Phi_{a,0}=(\pi/2-\beta)/n\), so after the rotation the two boundary
	phases are
	\[
		\phi_-
		=
		\frac{\pi}{2n}-\Phi_{a,0}
		=
		\frac1n\arcsin c,
		\qquad
		\phi_+
		=
		\frac{\pi}{2n}+\Phi_{a,0}
		=
		\frac1n\bigl(\pi-\arcsin c\bigr).
	\]
	These are precisely the two solutions of
	\(\sin(n\phi)=c\) in \(0<\phi<\pi/n\). Since
	\(\widetilde\varphi\) is strictly decreasing, it runs through the
	whole interval \([\phi_-,\phi_+]\). Moreover, on \(L_c\),
	\[
		\rho\leq1
		\quad\Longleftrightarrow\quad
		\sin(n\phi)\geq c
		\quad\Longleftrightarrow\quad
		\phi_-\leq\phi\leq\phi_+.
	\]
	Thus
	\(\widetilde\gamma([-s_a,s_a])
	=L_c\cap\overline{\mathbb B}^{2}\).
	Multiplying by \(\mathbb S^{n-1}\) gives
	\[
		e^{i\pi/(2n)}
		F_{a,0}\bigl([-s_a,s_a]\times\mathbb S^{n-1}\bigr)
		=
		M_c^n\cap\overline{\mathbb B}^{2n}.
	\]
	Since
	\(q_{a,0}=a^n=c\), the contact angles satisfy
	\[
		\cos\theta_+=c,
		\qquad
		\cos\theta_-=-c.
	\]

\begin{remark}
	The spherical factor \(\psi_{n-1}\) in
	Theorem~\ref{thm:embedded-self-similar-pieces} may be replaced by any compact
	minimal Legendrian immersion
	\[
		\psi:\Sigma^{n-1}\longrightarrow\mathbb S^{2n-1}.
	\]
	The map
	\[
		F_{a,\kappa}(s,p)
		=
		\gamma_{a,\kappa}(s)\psi(p)
	\]
	is then an exact Lagrangian immersion satisfying
	\(H+\kappa F_{a,\kappa}^{\perp}=0\), with boundary components
	\[
		\Gamma_\pm
		=
		e^{\pm i\Phi_{a,\kappa}}\psi(\Sigma).
	\]
	These are minimal Legendrian submanifolds with the contact angles
	given by \eqref{eq:contact-unified}, and the relative Liouville class
	is again nonzero.

	The resulting immersion need not be embedded. If \(\psi\) is embedded
	and
	\[
		\psi(\Sigma)\cap e^{it}\psi(\Sigma)=\varnothing
		\qquad
		\text{for all }0<|t|\leq2\Phi_{a,\kappa},
	\]
	then \(F_{a,\kappa}\) is embedded. For every compact embedded minimal
	Legendrian submanifold, this condition holds when \(a\) is sufficiently
	close to \(1\), since \(\Phi_{a,\kappa}\to0\) as \(a\to1\) and the
	vector field \(J\psi\) is transverse to \(\psi(\Sigma)\).
\end{remark}

\begin{remark}[Resolution of the classification problem]
\label{rem:subsequent-classification}
The preceding examples led us to ask whether every compact Lagrangian
self-similar immersion in the unit ball with two Legendrian capillary boundary components
is of Anciaux type. This problem has since been resolved as part of the
global classification in~\cite{GaoLuoMaYin}.
\end{remark}
	
\section*{Acknowledgments}
D. Gao is supported by the National Natural Science Foundation of China
(Grant No. 12401057) and the Beijing Natural Science Foundation
(Grant No. 1244039). H. Ma is supported by the National Natural Science
Foundation of China (Grant Nos. 12471048 and W2521103).
Z. Yao is supported by the National Natural Science Foundation of China
(Grant No. 12401061).


\normalsize\noindent
	\begin{thebibliography}{99}
		
		\bibitem{Anciaux}
		H. Anciaux,
		Construction of Lagrangian self-similar solutions to the mean curvature flow in \(\mathbb C^n\),
		\textit{Geom. Dedicata} \textbf{120} (2006), 37--48.
		
		\bibitem{CDGM}
		S. Cappell, D. DeTurck, H. Gluck and E. Y. Miller,
		Cohomology of harmonic forms on Riemannian manifolds with boundary,
		\textit{Forum Math.} \textbf{18} (2006), no. 6, 923--931.
		
		\bibitem{CastroLerma2}
		I. Castro and A. M. Lerma,
		The Clifford torus as a self-shrinker for the Lagrangian mean curvature flow,
		\textit{Int. Math. Res. Not. IMRN} \textbf{2014} (2014), 1515--1527.
		
		\bibitem{CastroUrbano}
		I. Castro and F. Urbano,
		On a minimal Lagrangian submanifold of \(\mathbb C^n\) foliated by spheres,
		\textit{Michigan Math. J.} \textbf{46} (1999), 71--82.
		
		\bibitem{ChauChenYuan}
		A. Chau, J. Chen and Y. Yuan,
		Rigidity of entire self-shrinking solutions to curvature flows,
		\textit{J. Reine Angew. Math.} \textbf{664} (2012), 229--239.
		
		\bibitem{ChengHoriWei}
		Q.-M. Cheng, H. Hori and G. Wei,
		Complete Lagrangian self-shrinkers in \(\mathbb R^4\),
		\textit{Math. Z.} \textbf{301} (2022), 3417--3468.
		
		\bibitem{DingXin}
		Q. Ding and Y. L. Xin,
		The rigidity theorems for Lagrangian self-shrinkers,
		\textit{J. Reine Angew. Math.} \textbf{692} (2014), 109--123.
			
		\bibitem{FraserSchoen}
		A. Fraser and R. Schoen,
		Uniqueness theorems for free boundary minimal disks in space forms,
		\textit{Int. Math. Res. Not. IMRN} \textbf{2015} (2015), 8268--8274.
		
		\bibitem{Gaia}
		F. Gaia,
		Free boundary Hamiltonian stationary Lagrangian discs in \(\mathbb C^2\),
		\textit{J. Geom. Anal.} \textbf{35} (2025), Paper No. 160.

		\bibitem{GaoLuoMaYin}
		D. Gao, Y. Luo, H. Ma and J. Yin,
		Classification of compact Lagrangian self-similar submanifolds with Legendrian capillary boundary in the unit ball,
		preprint, 2026.
			
		\bibitem{Gerner}
		W. Gerner,
		A unique continuation theorem for exterior differential forms on Riemannian manifolds with boundary,
		\textit{Differential Geom. Appl.} \textbf{103} (2026), Paper No. 102358.
		
		\bibitem{HarveyLawson}
		R. Harvey and H. B. Lawson, Jr.,
		Calibrated geometries,
		\textit{Acta Math.} \textbf{148} (1982), 47--157.
		
		\bibitem{JoyceLeeTsui}
		D. Joyce, Y.-I. Lee and M.-P. Tsui,
		Self-similar solutions and translating solitons for Lagrangian mean curvature flow,
		\textit{J. Differential Geom.} \textbf{84} (2010), 127--161.
		
		\bibitem{LeeWang1}
		Y.-I. Lee and M.-T. Wang,
		Hamiltonian stationary shrinkers and expanders for Lagrangian mean curvature flows,
		\textit{J. Differential Geom.} \textbf{83} (2009), 27--42.
		
		\bibitem{LeeWang2}
		Y.-I. Lee and M.-T. Wang,
		Hamiltonian stationary cones and self-similar solutions in higher dimension,
		\textit{Trans. Amer. Math. Soc.} \textbf{362} (2010), 1491--1503.
		
		\bibitem{LiWang}
		H. Li and X. Wang,
		New characterizations of the Clifford torus as a Lagrangian self-shrinker,
		\textit{J. Geom. Anal.} \textbf{27} (2017), 1393--1412.
		
		\bibitem{LiMM}
		M. M. C. Li,
		Free boundary minimal surfaces in the unit ball: recent advances and open questions,
		in \textit{Proceedings of the International Consortium of Chinese Mathematicians 2017},
		Int. Press, Boston, MA, 2020, pp. 401--435.
		
		\bibitem{LWW}
		M. Li, G. Wang and L. Weng,
		Lagrangian surfaces with Legendrian boundary,
		\textit{Sci. China Math.} \textbf{64} (2021), 1589--1598.
		
		\bibitem{LiWangWei}
		Z. Li, R. Wang and G. Wei,
		The rigidity theorem for complete Lagrangian self-shrinkers,
		\textit{J. Geom. Anal.} \textbf{35} (2025), Paper No. 69.
		
		\bibitem{LotayNeves}
		J. D. Lotay and A. Neves,
		Uniqueness of Lagrangian self-expanders,
		\textit{Geom. Topol.} \textbf{17} (2013), no. 5, 2689--2729.
		
		\bibitem{LS}
		Y. Luo and L. Sun,
		Rigidity theorems for minimal Lagrangian surfaces with Legendrian capillary boundary,
		\textit{Adv. Math.} \textbf{393} (2021), Paper No. 108124.
		
		\bibitem{Neves}
		A. Neves,
		Recent progress on singularities of Lagrangian mean curvature flow,
		in \textit{Surveys in Geometric Analysis and Relativity},
		Adv. Lect. Math. (ALM), vol. 20,
		Int. Press, Somerville, MA, 2011, pp. 413--438.
		
		\bibitem{Nitsche}
		J. C. C. Nitsche,
		Stationary partitioning of convex bodies,
		\textit{Arch. Ration. Mech. Anal.} \textbf{89} (1985), 1--19.
		
		\bibitem{RosUrbano}
		A. Ros and F. Urbano,
		Lagrangian submanifolds of \(\mathbb C^n\) with conformal Maslov form and the Whitney sphere,
		\textit{J. Math. Soc. Japan} \textbf{50} (1998), 203--226.
		
		\bibitem{RJ}
		R. Schoen and J. Wolfson,
		Minimizing area among Lagrangian surfaces: the mapping problem,
		\textit{J. Differential Geom.} \textbf{58} (2001), 1--86.
		
		
		
	\end{thebibliography}
\end{document}